\documentclass[preprint,11pt]{elsarticle}

\usepackage{amssymb,amsmath,amsfonts}
\usepackage[english]{babel}
\usepackage[letterpaper,top=2cm,bottom=2cm,left=3cm,right=3cm,marginparwidth=1.75cm]{geometry}
\usepackage{float}
\usepackage{graphicx}
\usepackage[colorlinks=true]{hyperref}
\usepackage{soul} 
\usepackage{tikz}

\newcommand{\io}{[0,1]}

\newcommand{\R}{\mathbb{R}}

\newcommand{\pt}{\text{ }\forall\text{ }}

\newcommand{\N}{\mathbb{N}}

\newtheorem{theorem}{Theorem}[section]
\newtheorem{definition}[theorem]{Definition}
\newtheorem{corollary}[theorem]{Corollary}
\newtheorem{proposition}[theorem]{Proposition}
\newtheorem{lemma}[theorem]{Lemma}
\newtheorem{example}[theorem]{Example}

\newenvironment{proof}[1][Proof]{\noindent{#1.} }{\ \rule{0.5em}{0.5em}}

\begin{document}
	
\begin{frontmatter}
\title{Sequential ordering relations {with application to} fuzzy numbers}
\author[label1]{Diego Garc\'ia-Zamora\corref{mycorrespondingauthor}}\ead{dgzamora@ujaen.es}
\author[label2]{Antonio Francisco Rold\'an L\'opez de Hierro}\ead{aroldan@ugr.es}
\cortext[mycorrespondingauthor]{Corresponding author}
\address[label1]{Department of Mathematics, Universidad de Ja\'en, Ja\'en, Spain.}
\address[label2]{Department of Statistics and Operations Research,  University of Granada, Granada, Spain.}

\begin{abstract}
The ranking of fuzzy numbers has become a challenging task in fuzzy set theory due to their complex, multi-dimensional nature. While the Klir-Yuan partial order provides a natural term-wise comparison of $\alpha$-cuts, it often leaves many fuzzy numbers incomparable. To address this, various ranking methods have been developed to construct total preorders between them. However, many classical approaches suffer from significant information loss as they imply a defuzzification process. On the other hand, approaches such as admissible orders allow defining total orders, but at the expense of imposing strict algebraic rules that may contradict human intuition. In this study, we introduce a generalized sequential ordering framework to overcome these limitations. By establishing a sequence space over a totally preordered base space, we construct a flexible lexicographical structure that sequentially resolves ties. We prove that this framework yields total preorders and, under injectivity conditions, total orders. Furthermore, we analyze the compatibility of these sequential orders with the notion of admissibility. We also show that our proposed framework provides a unified mathematical umbrella that encompasses and generalizes existing ranking techniques, offering highly discriminative ordering relations for fuzzy numbers and beyond.  
\end{abstract}


\end{frontmatter}

\section{Introduction}
\label{sec:intro}

Since the inception of fuzzy set theory by Zadeh \cite{zadeh1965}, modeling uncertainty has become a cornerstone in fields such as decision-making \cite{bellman1970}, approximate reasoning \cite{zadeh1975}, and expert systems \cite{zimmermann2012}. Fuzzy numbers, which generalize real numbers by considering imprecise boundaries, are particularly well-suited for capturing subjective or vague information \cite{DoCMF}. However, this flexibility introduces a fundamental challenge: unlike the crisp real numbers $\mathbb{R}$, which possess a natural total order, the set of fuzzy numbers does not inherently have a universally accepted linear ordering \cite{dubois1980, wang2001}.

The problem of ordering fuzzy numbers has attracted significant research interest over the decades, especially given its profound implications in fuzzy decision-making  \cite{carlsson1996fuzzy, kahraman2008fuzzy}. Early attempts to tackle this problem often relied on mapping fuzzy numbers to a single crisp representative or score. Prominent examples include the centroid or center-of-gravity method \cite{yager1978, Centroids} and other magnitude-based approaches  \cite{ABBASBANDY2009413}. While intuitively appealing and computationally straightforward, these defuzzification-based methods carry an essential limitation: they generally yield preorders rather than orders \cite{wang2001, Centroids}. That is, they lack the property of antisymmetry, i.e.,  distinct fuzzy numbers can map to the identical score, leading to ties and leaving them indistinguishable in ranking contexts.

To overcome this lack of discriminative power while preserving the natural geometric ordering of fuzzy numbers, researchers have turned to the concept of {admissible orders}. Building upon earlier work regarding linear orders on intervals \cite{AdmissibleInterval}, Zumelzu et al. \cite{zumelzu2022} introduced the notion of an admissible order on fuzzy numbers. An admissible order is a total order that refines the standard partial order introduced by Klir-Yuan \cite{Klir1994}. Admissible orders ensure that if one fuzzy number is strictly larger than another in the natural partial order sense, the admissible order will preserve this relationship.

Recently, the study of admissibility has seen rapid growth, with investigations into admissible Ordered Weighted Averaging operators \cite{AOWA} and the admissibility of centroid-based preorder \cite{Centroid_Admis} and the $\alpha$-order \cite{Alphaorder_Admis}. The $\alpha$-order, in particular, relies on evaluating fuzzy numbers at specific membership levels (or $\alpha$-cuts) to construct lexicographical hierarchies \cite{Neres2023}. This approach provides a technical method to generate total preorders by considering specific sequences of evaluation levels rather than just a single collapsed representative.

Motivated by these recent advancements and recognizing the utility of constructing robust ranking relations, this paper introduces a general framework for {sequential ordering relations} on fuzzy numbers. By extending the notions introduced in \cite{Neres2023} and leveraging structural properties of fuzzy number spaces, we propose a class of orderings that can be highly discriminative and interpretable at the same time. 

The remainder of this paper is organized as follows. Section \ref{sec:prelim} recalls the fundamental concepts regarding ordering relations and fuzzy numbers. Section \ref{sec:sequential} introduces the concept of sequential ordering relation and its main properties. In Section \ref{sec:alpha-order}, we prove that our proposal extends the notion of $\alpha$-order, and Section \ref{sec:unifying} shows that sequential orderings generalize many other classical ranking methods for fuzzy numbers, such as admissible orders and centroids. In Section \ref{sec:new}, we apply the notion of sequential ordering to obtain new ranking methods for different uncertainty models, such as finite fuzzy numbers and interval type-2 fuzzy numbers (IT2FN). Finally, Section \ref{sec:conclusion} presents some concluding remarks and directions for future work.

\section{Preliminaries}\label{sec:prelim}

Here, we present the main notions that constitute the basis of this manuscript. {Let $\mathbb{R}$ be the set of all real numbers, let $\mathbb{N}$ be the set of all positive integers, and let
$\mathbb{N}_{0}=\mathbb{N}\cup\{0\}$.}

\subsection{Ordering relations}

In this subsection, we recall the fundamental concepts regarding binary relations and orderings that will be used throughout this work. 

Let $X$ be a non-empty set. A {binary relation} on $X$ is a subset $\mathcal{R} \subseteq X \times X$. For simplicity, if $(x,y) \in \mathcal{R}$, we write $x \le y$ and refer to $\le$ as the binary relation on $X$. Two elements $x,y \in X$ are said to be {comparable} if either $x \le y$ or $y \le x$. A binary relation $\le$ on $X$ is said to be: (i) {reflexive} if $x \le x$ for all $x \in X$; (ii) {transitive} if $x \le y$ and $y \le z$ imply $x \le z$; (iii) {antisymmetric} if $x \le y$ and $y \le x$ imply $x = y$; (iv) {total} (or {linear}) if any two points $x,y \in X$ are comparable.

A relation satisfying reflexivity and transitivity is called a partial {preorder}. If a preorder is also antisymmetric, it is called a partial order. Partial preorders and partial orders that satisfy the totality property are referred to as {total preorders} and {total orders}, respectively. The trivial total preorder on a set $X$ is defined by $x \le y$ for all $x,y \in X$.

Given a preordered set $(K, \le)$, we can naturally associate with it an equivalence relation $\sim$ and a strict relation $<$ defined for any $x, y \in K$ as follows: i) $x\sim y \iff x \le y \text{ and } y \le x$; ii) $x < y \iff x \le y \text{ and } \text{not } (y \le x)$. Notice that if $(K, \le)$ is a totally ordered set (and hence antisymmetric), the equivalence relation $\sim$ reduces to the standard equality $=$.

To illustrate these concepts and provide context for the sequential orders we will construct later, let us recall the family of all non-empty, closed, and bounded subintervals of $\mathbb{R}$, denoted by $\operatorname{int}(\mathbb{R}) = \{[a,b] : a, b \in \mathbb{R}, a \le b\}$. Let $A = [a,b]$ and $B = [c,d]$ be two elements of $\operatorname{int}(\mathbb{R})$. The following are prominent examples of orderings defined on $\operatorname{int}(\mathbb{R})$ in the literature:

\begin{itemize}
    \item The {Kulisch-Miranker order} $\le_{KM}$ \cite{Kulisch1981}, which is a partial order but not total, defined by:
    \[ [a,b] \le_{KM} [c,d] \iff a \le c \text{ and } b \le d. \]
    \item The {Lexicographic 1 order} $\le_{Lex1}$, which is a total order defined by:
    \[ [a,b] \le_{Lex1} [c,d] \iff a < c \text{ or } (a = c \text{ and } b \le d). \]
    \item The {Lexicographic 2 order} $\le_{Lex2}$, which is a total  order defined by:
    \[ [a,b] \le_{Lex2} [c,d] \iff b < d \text{ or } (b = d \text{ and } a \le c). \]
    \item The {Xu and Yager order} $\le_{XY}$, which is a total  order defined by:
    \[ [a,b] \le_{XY} [c,d] \iff a+b < c+d \text{ or } (a+b = c+d \text{ and } b-a \le d-c). \]
    \item The {Twice Xu and Yager order} $\le_{2XY}$, which is a total order defined by:
    \[ [a,b] \le_{2XY} [c,d] \iff a+3b < c+3d \text{ or } (a+3b = c+3d \text{ and } b-a \le d-c). \]
\end{itemize}
All of these relations, with the exception of the Kulisch-Miranker order, are total orders \cite{zumelzu2022}.

\subsection{Fuzzy numbers and their ordering relations}

A fuzzy set on $\R$ is a mapping $A:\R\to\io$. For each $\alpha \in (0,1]$, the $\alpha$-cut
{of $A$} is the set $A_\alpha = \{x \in \mathbb{R} : A(x) \ge \alpha\}$,
{its support is the union of all its $\alpha$-cuts for $\alpha \in (0,1]$,}
whereas $A_0$ denotes the closure of the support of $A$ in the Euclidean topology of $\mathbb{R}$. 

A fuzzy number of $\R$ is a fuzzy set such that $A_1\neq \emptyset $, $A_0$ is a bounded {set}, and for $\alpha \in (0,1]$, $A_\alpha$ is a closed interval, i.e., $A_\alpha=[\,\underline{A}_{\, \alpha},\overline{A}_\alpha \,]$ (with $\underline{A}_{\,\alpha}\leq \overline{A}_\alpha$). Let $\operatorname{FN}(\mathbb{R})$ denote the set of all fuzzy numbers of the real line $\mathbb{R}$. The most natural way to compare two fuzzy numbers $A, B \in \operatorname{FN}(\mathbb{R})$ is through the Klir-Yuan partial order \cite{Klir1994} $\le_{KY}$, defined as:

\begin{equation*}
    A \le_{KY} B \iff \underline{A}_{\, \alpha} \le \underline{B}_{\, \alpha} \text{ and } \overline{A}_\alpha \le \overline{B}_\alpha \quad \forall \alpha \in [0,1].
\end{equation*}
Since $\le_{KY}$ is only a partial order, many fuzzy numbers are incomparable
{by $\leq_{KY}$}. To overcome this, researchers seek total orders that are `compatible'' with this natural partial order. This leads to the concept of admissibility:

\begin{definition}[\cite{Centroid_Admis}]
A total order (resp. preorder) $\le_{adm}$ on $\operatorname{FN}(\mathbb{R})$ is called an {admissible order (resp. preorder)} if it refines the Klir-Yuan partial order, i.e., for all $A, B \in \operatorname{FN}(\mathbb{R})$:
\[ A \le_{KY} B \implies A \le_{adm} B. \]
\end{definition}

One of the most classical ranking methods is the {centroid} (or center of gravity), which associates a fuzzy number $A$ with a crisp value $C(A) \in \mathbb{R}$. The centroid induces a total preorder $\le_{C}$ defined by $A \le_{C} B \iff C(A) \le C(B)$. While widely used, it is important to note that the centroid-based preorder is generally not an admissible preorder for the entire set $\operatorname{FN}(\mathbb{R})$ \cite{Centroid_Admis}.

On the other hand, the {$\alpha$-order} represents a different construction for total preorders on $\operatorname{FN}(\mathbb{R})$. Instead of a single crisp representative {value} like the centroid, the $\alpha$-order utilizes a sequence of evaluation levels. Given a sequence of values $\{\alpha_n\}_{n \in \mathbb{N}_0} \subset [0,1]$, two fuzzy numbers can be compared by some specific aggregations of their $\alpha_n$-cuts. This approach provides a systematic way to generate admissible total preorders \cite{Alphaorder_Admis}.

Finally, we recall that an IT2FN on $\R$ is defined as a mapping $A = (\underline{A}, \overline{A}) : \R \rightarrow [0, 1]\times [0,1]$, where both $\underline{A}, \overline{A}$ are fuzzy numbers and the inequality $\underline{A}(x) \leq \overline{A}(x)$
for all $x \in \R$ holds. The set of all the IT2FNs defined on $\mathbb{R}$ is denoted by $\operatorname{IFN}_2(\mathbb{R})$.

\subsection{{The }$\alpha${-order}}

In \cite{Neres2023}, the authors introduced a novel ranking methodology on a wide class of fuzzy sets (not necessarily fuzzy numbers) of the real line by employing membership values, weights, and two distribution functions as follows. 

{ In this subsection, let us fix $n\in\mathbb{N}$. We represent $\{1,2,\ldots,n\}$ by $J_{n}$, and $\{0,1,2,\ldots,n\}$ by $J_{n}^{0}$}.
A {membership degree vector} (or {membership
vector}) is a vector $\alpha=(\alpha_{1},\alpha_{2},\ldots,\alpha_{n}%
)\in[0,1]^{n}$, and a {weight vector} (or {weighting vector})
is $\omega=(\omega_{1},\omega_{2},\ldots,\omega_{n})\in[0,1]^{n}$ such
that $\omega_{1}+\omega_{2}+\ldots+\omega_{n}=1$. Henceforth, both vectors
will be considered jointly in a pair $(\alpha,\omega)$, so they will have the
same dimension ($n$) and, given $j\in J_{n}$, the value $\omega_{j}$ will be
called the {weight assigned to $\alpha_{j}$}. Let $W_{n}$ be the family
of all weight vectors in $[0,1]^{n}$, let $\Omega_{n}=(0,1]^{n}$ be the family
of all membership degree vectors (with non-null components), and let us denote
$\mathbf{0}_{n}=(0,0,\ldots,0)\in\mathbb{R}^{n}$. Notice that $\mathbf{0}%
_{n}\notin W_{n}$.

\begin{definition}[\cite{Neres2023}]
\label{def weight distribution functions}\noindent The {first weight
distribution} is the mapping $\vartheta:W_{n}\cup\{\mathbf{0}_{n}\}\rightarrow
W_{n}\cup\{\mathbf{0}_{n}\}$ given by:%
\[
\vartheta(\omega)(j)=\left\{
\begin{tabular}
[c]{ll}%
$0,$ & if $\omega_{j}=0$ or $j\in\eta(\omega),$\\[1mm]%
$\omega_{j}+\dfrac{s(\omega)}{\operatorname*{card}(\overline{\eta}(\omega))},$
& otherwise,
\end{tabular}
\right.
\]
and, given $\alpha\in\Omega_{n}$, the {second weight distribution} is the
mapping $\varphi_{\alpha}:W_{n}\cup\{\mathbf{0}_{n}\}\rightarrow W_{n}%
\cup\{\mathbf{0}_{n}\}$ given by:%
\[
\varphi_{\alpha}(\omega)(j)=\left\{
\begin{tabular}
[c]{ll}%
$0,$ & if $\omega_{j}=0$ or $j\in\theta_{\alpha}(\omega),$\\[1mm]%
$\omega_{j}+\dfrac{s_{\theta_{\alpha}}(\omega)}{\operatorname*{card}%
(\overline{\theta}_{\alpha}(\omega))},$ & otherwise
\end{tabular}
\right.
\]
($\vartheta(\omega)(j)$ is the $j$-th component of the vector $\vartheta
(\omega)$ and $\varphi_{\alpha}(\omega)(j)$ is the $j$-th component of the
vector $\varphi_{\alpha}(\omega)$).
\end{definition}
For the sake of brevity, we do not further develop the definition of the functions involved in the weight distribution. For the interested reader, they can be found in the original source \cite{Neres2023}. In order to follow this manuscript, it is enough to keep in mind that both mappings return either a weighting vector or the null vector.

\begin{definition}\label{def:class_C}
{(\cite{Neres2023})} Let $\mathcal{C}$ be the family of all fuzzy sets
$A\in {\operatorname{FS}}(\mathbb{R})$ satisfying:

\begin{itemize}
\item $A$ is normal, and its support is bounded;

\item for each $\alpha\in(0,1]$, the $\alpha$-cut $A_{\alpha}$ is a closed
subset of $\mathbb{R}$ with a finite number of connected components.
\end{itemize}
\end{definition}

Unless otherwise is stated, throughout this section $\mathcal{A}:\cup
_{n\in\mathbb{N}}\mathbb{R}^{n}\rightarrow[0,1]$ will always represent
{an extended} aggregation function on the real line, $A,B$ will be fuzzy
sets of the class $\mathcal{C}$ and $n\in\mathbb{N}$ will be the number of 
components of the membership vector $\alpha=(\alpha_{1},\alpha_{2}%
,\ldots,\alpha_{n})\in\Omega_{n}$ and the weight vector $\omega=(\omega
_{1},\omega_{2}\ldots,\omega_{n})\in W_{n}$.

\begin{definition}
{(\cite{Neres2023})} Given $k\in\mathbb{N}_{0}$, the $k$%
{-valuation of }$A$ {relative to} $\vartheta$ (respectively,
{to} $\varphi_{\alpha}$) is the real number:%
\[
v_{\alpha,\omega,\vartheta,\mathcal{A}}^{k}(A)=%
{\textstyle\sum\limits_{j=1}^{n}}
\,\left[ \,\vartheta^{k}(\omega)(j)\cdot\mathcal{A}_{\alpha_{j}}%
^{A}\,\right] \qquad\left( \text{respectively,\quad}v_{\alpha,\omega
,\varphi_{\alpha},\mathcal{A}}^{k}(A)=%
{\textstyle\sum\limits_{j=1}^{n}}
\,\left[ \,\varphi_{\alpha}^{k}(\omega)(j)\cdot\mathcal{A}_{\alpha_{j}}%
^{A}\,\right] \right) .
\]
\end{definition}

The previous definition led us to consider the $k${-valuation functions}
$v_{\alpha,\omega,\vartheta,\mathcal{A}}^{k},v_{\alpha,\omega,\varphi_{\alpha
},\mathcal{A}}^{k}:\mathcal{C}\rightarrow\mathbb{R}$. These functions can be
seen, for each $k\in\mathbb{N}_{0}$, as:%
\[
v_{\alpha,\omega,\vartheta,\mathcal{A}}^{k}(A)=\langle\,\vartheta^{k}%
(\omega),\mathcal{A}_{\alpha}^{A}\,\rangle\qquad\text{and}\qquad
v_{\alpha,\omega,\varphi_{\alpha},\mathcal{A}}^{k}(A)=\langle\,\varphi
_{\alpha}^{k}(\omega),\mathcal{A}_{\alpha}^{A}\,\rangle,
\]
{where $\langle\,\cdot\, ,\cdot\,\rangle$ represents the Euclidean scalar product of $\mathbb{R}^{n}$.}
When $k=0$, these functions satisfy that:%
\[
v_{\alpha,\omega,\vartheta,\mathcal{A}}^{0}(A)=v_{\alpha,\omega,\varphi
_{\alpha},\mathcal{A}}^{0}(A)=\langle\,\omega,\mathcal{A}_{\alpha}%
^{A}\,\rangle.
\]
{From now on, we assume that $\alpha$, $\omega$, and $\mathcal{A}$ are given in the context (and we will not mention them in definitions and statements).}
For simplicity, we denote $v_{\vartheta}^{k}=v_{\alpha,\omega,\vartheta,\mathcal{A}}^{k}$ and $v_{\varphi_{\alpha}}^{k}=v_{\alpha,\omega,\varphi_{\alpha},\mathcal{A}}^{k}$.

\begin{definition}
\label{def6.3}{(\cite{Neres2023})} Given $m\in J_{n-1}^{0}$, let
$\overset{\vartheta}{\underset{m}{=\joinrel=}}$, $\overset{\varphi_{\alpha}%
}{\underset{m}{=\joinrel=}}$ and $\overset{\alpha}{=}$ be the binary relations
on $\mathcal{C}$ defined as follows for $A,B\in\mathcal{C}$:%
\begin{align*}
& \bullet\quad A\overset{\vartheta}{\underset{m}{=\joinrel=}}B\quad
\text{iff}\quad v_{\vartheta}^{k}(A)=v_{\vartheta}^{k}(B)\text{ for each }k\in
J_{m}^{0};\\
& \bullet\quad A\overset{\varphi_{\alpha}}{\underset{m}{=\joinrel=}}%
B\quad\text{iff}\quad v_{\varphi_{\alpha}}^{k}(A)=v_{\varphi_{\alpha}}%
^{k}(B)\text{ for each }k\in J_{m}^{0};\\
& \bullet\quad A\overset{\alpha}{=}B\quad\text{iff}\quad A\overset{\vartheta
}{\underset{n-1}{=\joinrel=}}B\quad\text{and}\quad A\overset{\varphi_{\alpha}%
}{\underset{n-1}{=\joinrel=}}B.
\end{align*}
The relation \textquotedblleft$\overset{\alpha}{=}$\textquotedblright\ will be
called the $\alpha${-equality}.
\end{definition}

\begin{definition}
\label{def6.4}{(\cite{Neres2023})} Given $A,B\in\mathcal{C}$, we will
write:%
\begin{align*}
& \bullet~~A\overset{\vartheta}{<}B\quad\text{iff}\quad\left\{
\begin{tabular}
[c]{l}%
either$\quad v_{\vartheta}^{0}(A)<v_{\vartheta}^{0}(B)$\\
or$\quad\exists\,k_{0}\in J_{n-1}$ such that $A\overset{\vartheta}%
{\underset{m}{=\joinrel=}}B$ for each $m<k_{0}$\text{ and }$v_{\vartheta
}^{k_{0}}(A)<v_{\vartheta}^{k_{0}}(B);$%
\end{tabular}
\ \ \ \right. \\
& \bullet~~A\overset{\varphi_{\alpha}}{<}B\quad\text{iff}\quad\left\{
\begin{tabular}
[c]{l}%
either$\quad v_{\varphi_{\alpha}}^{0}(A)<v_{\varphi_{\alpha}}^{0}(B)$\\
or$\quad\exists\,k_{0}\in J_{n-1}$ such that $A\overset{\varphi_{\alpha}%
}{\underset{m}{=\joinrel=}}B$ for each $m<k_{0}$\text{ and }$v_{\varphi
_{\alpha}}^{k_{0}}(A)<v_{\varphi_{\alpha}}^{k_{0}}(B);$%
\end{tabular}
\ \ \ \right. \\
& \bullet~~A\overset{\alpha}{<}B\quad\text{iff}\quad A\overset{\vartheta}%
{<}B\quad\text{or}\quad\left( A\overset{\vartheta}{\underset{n-1}%
{=\joinrel=}}B\quad\text{and}\quad A\overset{\varphi_{\alpha}}{<}B\right) .
\end{align*}
The relation \textquotedblleft$\,\overset{\alpha}{<}\,$\textquotedblright%
\ will be called $\alpha${-minor}.
\end{definition}

Notice that $A\overset{\alpha}{<}B$ if, and only if,%
\[
\left\{
\begin{tabular}
[c]{l}%
either $\left\{
\begin{tabular}
[c]{l}%
$v_{\vartheta}^{0}(A)<v_{\vartheta}^{0}(B)\quad$\text{or}\\[2mm]%
$\exists\,k_{0}\in J_{n-1}$\text{ s.t. }$v_{\vartheta}^{j}(A)=v_{\vartheta
}^{j}(B),\forall j\in J_{k_{0}-1}^{0}$\text{ and }$v_{\vartheta}^{k_{0}%
}(A)<v_{\vartheta}^{k_{0}}(B),$%
\end{tabular}
\ \right. $\\[7mm]%
or\qquad$\left\{
\begin{tabular}
[c]{l}%
$v_{\vartheta}^{j}(A)=v_{\vartheta}^{j}(B),\forall j\in J_{n-1}^{0}\quad
$\text{and}\\[2mm]%
$\left\{
\begin{tabular}
[c]{l}%
$v_{\varphi_{\alpha}}^{0}(A)<v_{\varphi_{\alpha}}^{0}(B)\quad$\text{or}\\[2mm]%
$\exists\,k_{0}\in J_{n-1}$\text{ s.t. }$v_{\varphi_{\alpha}}^{j}%
(A)=v_{\varphi_{\alpha}}^{j}(B),\forall j\in J_{k_{0}-1}^{0}$\text{ and
}$v_{\varphi_{\alpha}}^{k_{0}}(A)<v_{\varphi_{\alpha}}^{k_{0}}(B).$%
\end{tabular}
\ \right. $%
\end{tabular}
\ \right. $%
\end{tabular}
\ \right.
\]

\begin{definition}
\label{H26 17 def alpha-order}Given $A,B\in\mathcal{C}$, we will write
$A\overset{\alpha}{\leq}B$ if either $A\overset{\alpha}{=}B$ or $A\overset
{\alpha}{<}B$.
\end{definition}
Note that the $\alpha$-order is indeed a preorder, while the alpha-equality is the corresponding equivalence relation \cite{Alphaorder_Admis}.
\subsection{Generation of admissible orders on fuzzy numbers}

As we said previously, admissible orders {for fuzzy numbers} are total orders that are compatible with the Klir-Yuan partial order. Below, we recall a mechanism to construct them.

\begin{definition}
\label{H26 12 def upper dense sequence}(\cite{WANG2014131}) A sequence $\{\alpha_{n}\}_{n\in\mathbb{N}} \subset (0,1]$ is called {upper dense} if for each $\beta \in (0,1]$ and each $\varepsilon > 0$, there is $m \in \mathbb{N}$ such that $\alpha_{m} \in [\beta, \beta+\varepsilon)$.
\end{definition}

Given an upper dense sequence $\{\alpha_{n}\} \subset (0,1]$ and $A, B \in \operatorname{FN}(\mathbb{R})$, it is possible to prove that $A \neq B$ if, and only if, there is $i \in \mathbb{N}$ such that $A_{\alpha_{i}} \neq B_{\alpha_{i}}$. This motivates a construction method for admissible orders.

\begin{definition}
\label{H26 15 def binary relation trianglelefteq}(\cite{zumelzu2022}) Given an upper dense sequence $\{\alpha_{n}\} \subset (0,1]$, let us define the mapping $m: \operatorname{FN}(\mathbb{R})\times \operatorname{FN}(\mathbb{R}) \rightarrow \mathbb{N}_{0}$ as:
\[
m(A,B) = \begin{cases} 
\, 0, & \text{if } A=B, \\
\, \min(\{i \in \mathbb{N} : A_{\alpha_{i}} \neq B_{\alpha_{i}}\}), & \text{if } A \neq B.
\end{cases}
\]
Let $\alpha=\{\alpha_{n}\} \subset (0,1]$ be an upper dense sequence and let $\preceq$ be a total order on $\operatorname*{int}(\mathbb{R})$ with strict relation $\prec$. We define on $\operatorname{FN}(\mathbb{R})$ the binary relation $\trianglelefteq$ as:
\begin{equation*}
A \trianglelefteq B \quad \text{if} \quad \text{( either } A=B \quad \text{or} \quad A_{\alpha_{m(A,B)}} \prec B_{\alpha_{m(A,B)}} \text{ ).}
\label{H26 15 def binary relation trianglelefteq, eq}
\end{equation*}
\end{definition}

\begin{theorem}
\label{H26 16 th trianglelefteq admissible order}(\cite{zumelzu2022}) If $\preceq$ is an admissible order on $\operatorname*{int}(\mathbb{R})$ and $\{\alpha_{n}\}$ is an upper dense sequence in $(0,1]$, then the binary relation $\trianglelefteq$ is an admissible order on $\operatorname{FN}(\mathbb{R})$.
\end{theorem}

\section{Sequential ordering relations}
\label{sec:sequential}

The fundamental challenge in comparing complex mathematical objects lies in capturing their multidimensional nature without reducing them to overly simplistic single-value metrics. Classical ranking methods frequently suffer from severe information loss, leading to total preorders with broad equivalence classes. On the other hand, recent approaches such as admissible orders based on dense sequences \cite{zumelzu2022} or alpha orders \cite{Neres2023}, rely on sophisticated approaches that become highly technical. To overcome these limitations, this section introduces the core theoretical contribution of this manuscript: the generalized sequential ordering framework. By establishing a sequence space $\operatorname{seq}(K)$ built over a totally preordered base space $(K, \le)$, and defining an evaluation mapping $\Phi$, we construct a flexible lexicographical structure to make comparisons. This framework is designed to sequentially resolve ties, providing a unified mathematical umbrella that naturally absorbs existing ranking methods for fuzzy numbers.
\subsection{Orderings on sequences of a totally preordered set}

Let $(K, \lesssim)$ be a totally preordered set. We denote by $\operatorname{seq}(K)$ the family of all sequences of elements in $K$ indexed on $\N_{0}$, that is:
$$\operatorname{seq}(K) = K^{\N_{0}} = \{\{a_n\}_{n\in\N_{0}} : a_n \in K \text{ for each } n \in \N_{0}\}.$$

For simplicity, we denote a sequence $\{a_n\}_{n\in\N_{0}} \in \operatorname{seq}(K)$ by $\mathfrak{a} = \{a_n\}$. Given $\mathfrak{a} = \{a_n\}, \mathfrak{b} = \{b_n\} \in \operatorname{seq}(K)$, we will write $\mathfrak{a} \equiv \mathfrak{b}$ when $a_n \sim b_n$ for all $n \in \N_{0}$.

\begin{definition}\label{def:order_for_sequences}
Given $\mathfrak{a} = \{a_n\}$ and $\mathfrak{b} = \{b_n\} \in \operatorname{seq}(K)$, we will write $\mathfrak{a} \sqsubseteq \mathfrak{b}$ when one, and only one, of the following conditions holds:
\begin{enumerate}
    \item $\mathfrak{a} \equiv \mathfrak{b}$,
    \item there is $n_0 \in \N_{0}$ such that $a_n \sim b_n$ for each $n \in J_{n_0-1}^0$ and $a_{n_0} \prec b_{n_0}$.
\end{enumerate}
(If $n_0=0$, we agree that the condition $a_n \sim b_n$ for $n \in J_{-1}^0$ is empty, meaning simply $a_0 \prec b_0$).
\end{definition}

\begin{theorem}\label{th:preorder_sucesiones}
If $(K, \lesssim)$ is a totally preordered set, the binary relation $\sqsubseteq$ is a total preorder on $\operatorname{seq}(K)$.
\end{theorem}

\begin{proof}
The reflexivity of $\sqsubseteq$ immediately follows from the reflexivity of $\sim$.

To show the transitivity, let $\mathfrak{a}=\{a_{n}\}$, $\mathfrak{b}=\{b_{n}\}$, $\mathfrak{c}=\{c_{n}\}\in \operatorname{seq}(K)$ be such that $\mathfrak{a}\sqsubseteq\mathfrak{b}$ and $\mathfrak{b}\sqsubseteq\mathfrak{c}$.
If $\mathfrak{a}\equiv\mathfrak{b}$ or $\mathfrak{b}\equiv\mathfrak{c}$, it trivially follows that $\mathfrak{a}\sqsubseteq\mathfrak{c}$ because $\sim$ is transitive and preserves strict inequalities. Next, suppose $\mathfrak{a}\not\equiv\mathfrak{b}$ and $\mathfrak{b}\not\equiv\mathfrak{c}$. This implies there exist indices $n_0, n_1 \in \N_0$ such that:
\begin{itemize}
    \item $a_n \sim b_n$ for $n < n_0$ and $a_{n_0} \prec b_{n_0}$,
    \item $b_n \sim c_n$ for $n < n_1$ and $b_{n_1} \prec c_{n_1}$.
\end{itemize}
Let $n_2 = \min\{n_0, n_1\}$. Because $\sim$ is an equivalence relation, for all $n < n_2$, we have $a_n \sim b_n \sim c_n$, so $a_n \sim c_n$. 
If $n_2 = n_0 < n_1$, then $a_{n_0} \prec b_{n_0} \sim c_{n_0}$, which implies $a_{n_0} \prec c_{n_0}$. 
If $n_2 = n_1 < n_0$, then $a_{n_1} \sim b_{n_1} \prec c_{n_1}$, which implies $a_{n_1} \prec c_{n_1}$. 
If $n_2 = n_0 = n_1$, then $a_{n_2} \prec b_{n_2} \prec c_{n_2}$, implying $a_{n_2} \prec c_{n_2}$. 
In all cases, $\mathfrak{a}\sqsubseteq\mathfrak{c}$.

Finally, to study the totality, let $\mathfrak{a}=\{a_{n}\}$, $\mathfrak{b}=\{b_{n}\}\in \operatorname{seq}(K)$. If $\mathfrak{a}\equiv\mathfrak{b}$, then $\mathfrak{a} \sqsubseteq \mathfrak{b}$. Suppose $\mathfrak{a}\not\equiv\mathfrak{b}$. There exists a minimum index $n_0=\min(\{n\in\N_{0} : \text{not }(a_{n} \sim b_{n})\})$. Since $K$ is a
{totally preordered set}, for this specific index $n_0$, either $a_{n_0} \prec b_{n_0}$ (implying $\mathfrak{a}\sqsubseteq\mathfrak{b}$) or $b_{n_0} \prec a_{n_0}$ (implying $\mathfrak{b}\sqsubseteq\mathfrak{a}$). Therefore, the sequences are always comparable, and the relation is total.
\end{proof}

When the base space $K$ is not just a total {preordered set}, but a totally {ordered set}, the equivalence relation $\sim$ reduces to strict equality ($=$). In this scenario, the sequential order on $\operatorname{seq}(K)$ gains the {antisymmetric} property.

\begin{corollary}\label{coro:order_sucesiones}
If $(K, \preceq)$ is a totally ordered set, then the binary relation $\sqsubseteq$ is a total order on $\operatorname{seq}(K)$.
\end{corollary}

\begin{proof}
By Theorem \ref{th:preorder_sucesiones}, $\sqsubseteq$ is a total preorder on $\operatorname{seq}(K)$. It only remains to prove antisymmetry. Let $\mathfrak{a}=\{a_{n}\}$ and $\mathfrak{b}=\{b_{n}\}\in \operatorname{seq}(K)$ be such that $\mathfrak{a}\sqsubseteq\mathfrak{b}$ and $\mathfrak{b}\sqsubseteq\mathfrak{a}$. To obtain a contradiction, suppose that $\mathfrak{a} \neq \mathfrak{b}$. Since $K$ is a totally ordered set, $\mathfrak{a} \neq \mathfrak{b}$ implies $\mathfrak{a} \not\equiv \mathfrak{b}$. Thus, there is a minimum index $n_0 \in \N_0$ such that $a_{n_0} \neq b_{n_0}$. Because $K$ is totally ordered, either $a_{n_0} \prec b_{n_0}$ or $b_{n_0} \prec a_{n_0}$. 
If $a_{n_0} \prec b_{n_0}$, the condition $\mathfrak{b}\sqsubseteq\mathfrak{a}$ cannot hold. If $b_{n_0} \prec a_{n_0}$, the condition $\mathfrak{a}\sqsubseteq\mathfrak{b}$ cannot hold. This is a contradiction. Therefore, $\mathfrak{a}=\mathfrak{b}$, and $\sqsubseteq$ is antisymmetric.
\end{proof}



\subsection{Ordering relations induced by sequences}

In this section, we introduce a broad family of binary relations on arbitrary sets that behave as total preorders and, under specific conditions, become total orders. 

\begin{definition}\label{def:seq_bin_rel}
Given a non-empty set $X$, a totally preordered set $(K, \preceq)$, and a mapping $\Phi: X \rightarrow \operatorname{seq}(K)$, we define the {$\Phi$-sequential binary relation} $\lesssim_\Phi$ on $X$ such that, for each $x, y \in X$:
$$x \lesssim_\Phi y \quad \text{if} \quad \Phi(x) \sqsubseteq \Phi(y).$$
\end{definition}

\begin{theorem}\label{th:Seq_Preorder}
Given a non-empty set $X$, a totally preordered set $(K, \preceq)$, and a mapping $\Phi: X \rightarrow \operatorname{seq}(K)$, the $\Phi$-sequential binary relation $\lesssim_\Phi$ is a total preorder on $X$.
\end{theorem}

\begin{proof}
Let $x, y, z \in X$. From Theorem \ref{th:preorder_sucesiones}, we know that the binary relation $\sqsubseteq$ is a total preorder on $\operatorname{seq}(K)$. Consequenly, 
\begin{enumerate}
    \item {Reflexivity:} Since $\sqsubseteq$ is reflexive on $\operatorname{seq}(K)$, we have $\Phi(x) \sqsubseteq \Phi(x)$ for all $x \in X$. Thus, $x \lesssim_\Phi x$.
    \item {Transitivity:} Suppose $x \lesssim_\Phi y$ and $y \lesssim_\Phi z$. By definition, $\Phi(x) \sqsubseteq \Phi(y)$ and $\Phi(y) \sqsubseteq \Phi(z)$. Since $\sqsubseteq$ is transitive on $\operatorname{seq}(K)$, it follows that $\Phi(x) \sqsubseteq \Phi(z)$, which implies $x \lesssim_\Phi z$.
    \item {Totality:} For any $x, y \in X$, the sequences $\Phi(x)$ and $\Phi(y)$ are elements of $\operatorname{seq}(K)$. Because $\sqsubseteq$ is total on $\operatorname{seq}(K)$, either $\Phi(x) \sqsubseteq \Phi(y)$ or $\Phi(y) \sqsubseteq \Phi(x)$. Consequently, either $x \lesssim_\Phi y$ or $y \lesssim_\Phi x$, proving that $\lesssim_\Phi$ is total on $X$.
\end{enumerate}
\end{proof}

To recover the antisymmetry property necessary for a partial order, we must restrict our base space to a totally ordered set and require the mapping $\Phi$ to separate distinct elements.

\begin{corollary}\label{coro:Seq_Order}
Let $X$ be a non-empty set, $(K, \preceq)$ a totally ordered set, and $\Phi: X \rightarrow \operatorname{seq}(K)$ an injective mapping. Then the $\Phi$-sequential binary relation $\preceq_\Phi$ is a total order on $X$.
\end{corollary}

\begin{proof}
Since $(K, \preceq)$ is a totally ordered set, its induced relation $\sqsubseteq$ on $\operatorname{seq}(K)$ becomes a total order. By Theorem \ref{th:Seq_Preorder}, $\lesssim_\Phi$ is already reflexive, transitive, and total. 

To prove
{the}
antisymmetry, let $x, y \in X$ be such that $x \preceq_\Phi y$ and $y \preceq_\Phi x$. By definition, this implies $\Phi(x) \sqsubseteq \Phi(y)$ and $\Phi(y) \sqsubseteq \Phi(x)$. Given that $\sqsubseteq$ is antisymmetric on $\operatorname{seq}(K)$, it follows that $\Phi(x) = \Phi(y)$. Finally, because $\Phi$ is an injective mapping, $\Phi(x) = \Phi(y)$ implies $x = y$. Thus, $\preceq_\Phi$ is antisymmetric, making it a total order.
\end{proof}

\subsection{Algebraic properties of sequential orders}\label{sec:algebraic}

A fundamental requirement for any practical ordering framework is its compatibility with underlying algebraic structures. In this section, we generalize the analysis of algebraic compatibility to our sequential framework. We establish the conditions under which the $\Phi$-sequential binary relation $\le_\Phi$ on an arbitrary set $X$ preserves internal binary operations (such as addition) and external operations (such as scalar multiplication).

Let $X$ be a non-empty set endowed with a binary operation $\oplus: X \times X \to X$. A binary relation $\le$ on $X$ is said to be {compatible with $\oplus$} if for all $x, y, z \in X$:
\[ x \le y \implies x \oplus z \le y \oplus z \quad \text{and} \quad z \oplus x \le z \oplus y. \]

To analyze the sequential relation $\le_\Phi$, we must first evaluate the behavior of the sequence space itself. Let $(K, \le, +)$ be a totally preordered set endowed with a binary operation $+$. We naturally define the term-wise operation on $\operatorname{seq}(K)$ as $\mathfrak{a} +_{\operatorname*{seq}} \mathfrak{b} = \{a_n + b_n\}_{n \in \mathbb{N}_{0}}$.
{The operation $+$ is said to be strictly translation invariant if for all $u,v,w\in K$ such that $u\sim v$, it holds $u+w\sim v+w\text{ and }w+u\sim w+v$, and $u<v$ implies that $u+w<v+w$ and $w+u<w+v$.}

\begin{lemma}
\label{lemma_seq_compatibility}
Let $(K, \le, +)$ be a totally preordered set where the operation $+$ is strictly translation invariant.
Then, the sequential total preorder $\sqsubseteq$ is compatible with $+_{\operatorname*{seq}}$ on $\operatorname{seq}(K)$. That is, for any $\mathfrak{a}, \mathfrak{b}, \mathfrak{c} \in \operatorname{seq}(K)$:
\[ \mathfrak{a} \sqsubseteq \mathfrak{b} \implies ( \, \mathfrak{a} +_{\operatorname*{seq}} \mathfrak{c} \sqsubseteq \mathfrak{b} +_{\operatorname*{seq}} \mathfrak{c} \quad \text{and} \quad \mathfrak{c} +_{\operatorname*{seq}} \mathfrak{a} \sqsubseteq \mathfrak{c} +_{\operatorname*{seq}} \mathfrak{b} \,). \]
\end{lemma}

\begin{proof}
We prove the right-sided compatibility (the left-sided proof is analogous). Assume $\mathfrak{a} \sqsubseteq \mathfrak{b}$. According to Definition \ref{def:order_for_sequences}, there are two possibilities:
\begin{enumerate}
    \item $\mathfrak{a} \equiv \mathfrak{b}$, meaning $a_n \sim b_n$ for all $n \in \mathbb{N}_{0}$. By the translation invariance of $\sim$ on $K$, we have $a_n + c_n \sim b_n + c_n$ for all $n \in \mathbb{N}_{0}$. Thus, $\mathfrak{a} +_{\operatorname*{seq}} \mathfrak{c} \equiv \mathfrak{b} +_{\operatorname*{seq}} \mathfrak{c}$, which trivially implies $\mathfrak{a} +_{\operatorname*{seq}} \mathfrak{c} \sqsubseteq \mathfrak{b} +_{\operatorname*{seq}} \mathfrak{c}$.
    
    \item There exists a minimum index $n_0 \in \mathbb{N}_{0}$ such that $a_n \sim b_n$ for all $n < n_0$, and $a_{n_0} < b_{n_0}$. By the properties of $K$, for all $n < n_0$, $a_n + c_n \sim b_n + c_n$. Furthermore, the strict relation is preserved at the divergence index $n_0$, so $a_{n_0} + c_{n_0} < b_{n_0} + c_{n_0}$. Therefore, by the definition of the sequence order, $\mathfrak{a} +_{\operatorname*{seq}} \mathfrak{c} \sqsubseteq \mathfrak{b} +_{\operatorname*{seq}} \mathfrak{c}$.
\end{enumerate}
In both cases, compatibility holds.
\end{proof}

With the algebraic stability of the sequence space guaranteed, the compatibility of the generalized $\Phi$-sequential order $\le_\Phi$ depends entirely on whether the mapping $\Phi$ preserves the algebraic structure of the set $X$.

\begin{definition}
Let $(X, \oplus)$ and $(K, +)$ be sets endowed with binary operations. A mapping $\Phi: X \to \operatorname{seq}(K)$ is said to {preserve the operation} (or to be a homomorphism) if, for all $x, y \in X$:
\[ \Phi(x \oplus y) \equiv \Phi(x) +_{\operatorname*{seq}} \Phi(y). \]
\end{definition}

\begin{theorem}
\label{th_algebraic_compatibility}
Let $(X, \oplus)$ be a set with a binary operation, and let $(K, \le, +)$ be a totally preordered set where $+$ is strictly translation invariant. If the mapping $\Phi: X \to \operatorname{seq}(K)$ preserves the operation, then the $\Phi$-sequential binary relation $\le_\Phi$ is compatible with $\oplus$ on $X$.
\end{theorem}

\begin{proof}
Let $x, y, z \in X$ and assume $x \le_\Phi y$. By definition, this means $\Phi(x) \sqsubseteq \Phi(y)$. Since $\Phi$ preserves the operation, we have $\Phi(x \oplus z) \equiv \Phi(x) +_{\operatorname*{seq}} \Phi(z)$ and $\Phi(y \oplus z) \equiv \Phi(y) +_{\operatorname*{seq}} \Phi(z)$. By Lemma \ref{lemma_seq_compatibility}, the condition $\Phi(x) \sqsubseteq \Phi(y)$ implies $\Phi(x) +_{\operatorname*{seq}} \Phi(z) \sqsubseteq \Phi(y) +_{\operatorname*{seq}} \Phi(z)$. Since the equivalence relation $\equiv$ maintains the order $\sqsubseteq$, we conclude that $\Phi(x \oplus z) \sqsubseteq \Phi(y \oplus z)$, which exactly means $x \oplus z \le_\Phi y \oplus z$. The proof for $z \oplus x \le_\Phi z \oplus y$ follows identical logic.
\end{proof}

This algebraic preservation easily extends to external operations. Let $\Omega$ be a set of operators (scalars), and let $\odot: \Omega \times X \to X$ and $\cdot : \Omega \times K \to K$ be external operations. We say $\le_\Phi$ is compatible with $\odot$ if $x \le_\Phi y \implies \lambda \odot x \le_\Phi \lambda \odot y$ for all $\lambda \in \Omega$. 

\begin{theorem}
\label{LM_algebraic_compatibility_2}Let the external operation on $K$ preserve the order (i.e., $u \le v \implies \lambda \cdot u \le \lambda \cdot v$). If the mapping $\Phi$ preserves the external operation such that $\Phi(\lambda \odot x) \equiv \lambda \cdot_{\operatorname{seq}} \Phi(x)$ for all $x \in X$ and $\lambda \in \Omega$, then the $\Phi$-sequential binary relation $\le_\Phi$ is compatible with the external operation $\odot$ on $X$.
\end{theorem}

\section{Sequential orderings and the $\alpha$-order}\label{sec:alpha-order}

The $\alpha$-order relies on a sequence of membership degrees, weight vectors, and valuation functions ($v_{\vartheta}^{k}$ and $v_{\varphi_{\alpha}}^{k}$). Because this relation breaks ties sequentially, it is a perfect candidate for absorption into our unified sequential ordering framework. Here, we demonstrate that the complex admissible preorder defined on the class $\mathcal{C}$ (see \ref{def:class_C}) is, in fact, a specific instance of a $\Phi$-sequential order. In this scenario, the totally ordered base space $K$ is the set of real numbers $\mathbb{R}$ equipped with its standard total order.

\begin{theorem}
\label{H26 05 lem alpha-order sequential}
Given $n\in\mathbb{N}$, a weight vector $\omega=(\omega_{1},\omega_{2},\ldots,\omega_{n})\in W_{n}$, a membership degree vector $\alpha=(\alpha_{1},\alpha_{2},\ldots,\alpha_{n})\in(0,1]^{n}$ and
{an extended}
aggregation function $\mathcal{A}:\cup_{m\in\mathbb{N}}\mathbb{R}^{m}\rightarrow\mathbb{R}$, the $\alpha$-order $\overset{\alpha}{\leq}$ is exactly the $\Phi_{\alpha,\omega,\mathcal{A}}$-sequential binary relation on {the class $\mathcal{C}$} associated to the mapping $\Phi_{\alpha,\omega,\mathcal{A}}:\mathcal{C}\rightarrow \operatorname{seq}(\mathbb{R})$ defined, for each $A\in\mathcal{C}$, by:
\[
\Phi_{\alpha,\omega,\mathcal{A}}(A)(k)=\left\{
\begin{tabular}
[c]{ll}%
$v_{\vartheta}^{k}(A),$ & if $~k\in\{0,1,\ldots,n-1\},$\\[4mm]%
$v_{\varphi_{\alpha}}^{k-(n-1)}(A),$ & if $~k\in\{n,n+1,\ldots,2n-2\},$\\[3mm]%
$0,$ & if $~k\geq 2n-1.$%
\end{tabular}
\right.
\]
\end{theorem}

\begin{proof}
By Definition \ref{def:seq_bin_rel}, the sequential binary relation is given by $A \le_{\Phi_{\alpha,\omega,\mathcal{A}}} B \iff \Phi_{\alpha,\omega,\mathcal{A}}(A) \sqsubseteq \Phi_{\alpha,\omega,\mathcal{A}}(B)$. We analyze the conditions of the total preorder $\sqsubseteq$ on $\operatorname{seq}(\mathbb{R})$:
\begin{itemize}
    \item Equivalence. $\Phi_{\alpha,\omega,\mathcal{A}}(A) = \Phi_{\alpha,\omega,\mathcal{A}}(B)$ if and only if all their sequence components are identical. This means $v_{\vartheta}^{k}(A) = v_{\vartheta}^{k}(B)$ for all $k \in \{0, \dots, n-1\}$, which exactly corresponds to $A\overset{\vartheta}{\underset{n-1}{=\joinrel=}}B$. It also means $v_{\varphi_{\alpha}}^{j}(A) = v_{\varphi_{\alpha}}^{j}(B)$ for all $j \in \{1, \dots, n-1\}$. Since $v_{\vartheta}^{0} = v_{\varphi_{\alpha}}^{0}$, this covers all components up to $n-1$, corresponding to $A\overset{\varphi_{\alpha}}{\underset{n-1}{=\joinrel=}}B$. Both conditions holding simultaneously is precisely the definition of the $\alpha$-equality, $A \overset{\alpha}{=} B$.
    \item Strict Order in the first half. Suppose the sequences differ at a minimum index $k_0 < n$. This implies $v_{\vartheta}^{k}(A) = v_{\vartheta}^{k}(B)$ for all $k < k_0$ and $v_{\vartheta}^{k_0}(A) < v_{\vartheta}^{k_0}(B)$. This is the exact definition of $A \overset{\vartheta}{<} B$. By definition of the $\alpha$-minor relation, this implies $A \overset{\alpha}{<} B$.
    \item Strict Order in the second half. Suppose the sequences differ at a minimum index $k_0 \in \{n, \dots, 2n-2\}$. This means the first $n$ elements (indices $0$ to $n-1$) are identical, so $A\overset{\vartheta}{\underset{n-1}{=\joinrel=}}B$. Let $j_0 = k_0 - (n-1)$. The difference at $k_0$ implies $v_{\varphi_{\alpha}}^{j}(A) = v_{\varphi_{\alpha}}^{j}(B)$ for all $j < j_0$ and $v_{\varphi_{\alpha}}^{j_0}(A) < v_{\varphi_{\alpha}}^{j_0}(B)$. This is exactly the definition of $A \overset{\varphi_{\alpha}}{<} B$. Combined with the $\vartheta$-equality, this gives $A \overset{\alpha}{<} B$. 
\end{itemize}

In all cases, the behavior of the sequential order $\sqsubseteq$ perfectly mirrors the piecewise lexicographic conditions of the $\alpha$-order. Therefore, $A \le_{\Phi_{\alpha,\omega,\mathcal{A}}} B \iff A \overset{\alpha}{\leq} B$.
\end{proof}

Furthermore, the intermediate binary relations $\overset{\vartheta}{\leq}$ and $\overset{\varphi_{\alpha}}{\leq}$, which evaluate only one of the weight distributions, can also be represented as sequential orders associated with the following truncated mappings:

\begin{itemize}
\item $\Phi_{\vartheta}:\mathcal{C}\rightarrow \operatorname{seq}(\mathbb{R})$ defined, for each $A\in\mathcal{C}$, by:
\[
\Phi_{\vartheta}(A)(k)=\left\{
\begin{tabular}
[c]{ll}%
$v_{\vartheta}^{k}(A),$ & if $~k\in\{0,1,\ldots,n-1\},$\\[3mm]%
$0,$ & if $~k\geq n.$%
\end{tabular}
\right.
\]

\item $\Phi_{\varphi_{\alpha}}:\mathcal{C}\rightarrow \operatorname{seq}(\mathbb{R})$ defined, for each $A\in\mathcal{C}$, by:
\[
\Phi_{\varphi_{\alpha}}(A)(k)=\left\{
\begin{tabular}
[c]{ll}%
$v_{\varphi_{\alpha}}^{k}(A),$ & if $~k\in\{0,1,\ldots,n-1\},$\\[3mm]%
$0,$ & if $~k\geq n.$%
\end{tabular}
\right.
\]
\end{itemize}

\section{Sequential ordering relations as a unified framework for existing orders for fuzzy numbers}\label{sec:unifying}

In this section, we study how the $\Phi$-sequential binary relation introduced before generalizes the main existing ranking approaches for fuzzy numbers. By appropriately selecting the totally preordered base space $(K, \le)$ and the mapping $\Phi: \operatorname{FN}(\mathbb{R}) \to \operatorname{seq}(K)$, we can recover and extend many of the standard ordering methods for fuzzy numbers found in the literature. 

For the majority classical methods, the target space $K$ is simply the set of real numbers $\mathbb{R}$ equipped with its standard total order $\le$. In this case, the equivalence relation $\sim$ is the standard equality $=$, and the strict relation $<$ is the standard strict inequality. We begin by showing that the sequential framework naturally embeds the simpler subfamilies of fuzzy numbers.

\begin{example}[Crisp fuzzy numbers]

Given $r \in \mathbb{R}$, we denote by $\tilde{r}$ to the real function that associates $1$ to $r$, and $0$ to any other real number. Then $\tilde{r}$ is a fuzzy number, known as {crisp} fuzzy number. We denote by $\widetilde{\mathbb{R}}$ to the family of all crisp fuzzy numbers (because it is biyective to $\mathbb{R}$). Define the mapping $\Phi_{crisp}: \widetilde{\mathbb{R}} \to \operatorname{seq}(\mathbb{R})$ as:
\[ \Phi_{crisp}(\tilde{r}) = \{r, 0, 0, \dots\}. \]
Let $\tilde{r_1}, \tilde{r_2} \in \widetilde{\mathbb{R}}$. Applying Definition \ref{def:seq_bin_rel}, we have $\tilde{r_1} \le_{\Phi_{crisp}} \tilde{r_2}$ if and only if $\Phi_{crisp}(\tilde{r_1}) \sqsubseteq \Phi_{crisp}(\tilde{r_2})$. By the definition of $\sqsubseteq$ on sequences, this occurs if and only if either the sequences are identical ($r_1 = r_2$) or they differ at the first index ($r_1 < r_2$). Therefore, $\tilde{r_1} \le_{\Phi_{crisp}} \tilde{r_2} \iff r_1 \le r_2$, which perfectly recovers the standard order of the real line.
\end{example}
The framework can also be adapted to compare families of fuzzy numbers defined by a fixed number of parameters, such as the family of trapezoidal fuzzy numbers $\operatorname{TrFN}(\mathbb{R})$. A trapezoidal fuzzy number $A$ can be represented by a 4-tuple $A = (a, b, c, d)$ with $a \le b \le c \le d$ \cite{Centroid_Admis}.
\begin{proposition}[Trapezoidal fuzzy numbers]
We can define a sequential order on $\operatorname{TrFN}(\mathbb{R})$ via the mapping $\Phi_{trap}: \operatorname{TrFN}(\mathbb{R}) \to \operatorname{seq}(\mathbb{R})$ defined as:
\[ \Phi_{trap}(a, b, c, d) = \{a, b, c, d, 0, 0, \dots\}. \]

\end{proposition}
This sequential order prioritizes the left-most point of the support, breaking ties with the core boundaries, and finally the right-most point of the support. By permuting the parameters in the sequence or using linear combinations (such as the expected value), we can immediately generate a vast family of total orders tailored to specific decision-making contexts.

A vast portion of the literature on ranking fuzzy numbers relies on ranking indices (or defuzzification functions). A ranking index is a mapping $I: \operatorname{FN}(\mathbb{R}) \to \mathbb{R}$, which induces a total preorder $\le_I$ on $\operatorname{FN}(\mathbb{R})$ defined by $A \le_I B \iff I(A) \le I(B)$. The centroid approach discussed earlier is an example of a ranking index \cite{Centroid_Admis}.

\begin{proposition}
Each total preorder defined by a ranking index $I: \operatorname{FN}(\mathbb{R}) \to \mathbb{R}$ is a sequential preorder.
\end{proposition}

\begin{proof}
Let $I: \operatorname{FN}(\mathbb{R}) \to \mathbb{R}$ be a ranking index. Define the mapping $\Phi_I: \operatorname{FN}(\mathbb{R}) \to \operatorname{seq}(\mathbb{R})$ by:
\[ \Phi_I(A) = \{I(A), 0, 0, \dots\} \quad \text{for all } A \in \operatorname{FN}(\mathbb{R}). \]
Let $A, B \in \operatorname{FN}(\mathbb{R})$. According to the definition of the sequential binary relation, $A \le_{\Phi_I} B \iff \Phi_I(A) \sqsubseteq \Phi_I(B)$.
Since the base set is $\mathbb{R}$, $\Phi_I(A) \sqsubseteq \Phi_I(B)$ holds if and only if either: $\Phi_I(A) = \Phi_I(B)$, which implies $I(A) = I(B)$; or the sequences differ at the first index ($n_0 = 0$), meaning $I(A) < I(B)$.

These two conditions are jointly equivalent to $I(A) \le I(B)$. Therefore, $A \le_{\Phi_I} B \iff A \le_I B$, proving that $\le_I$ is the $\Phi_I$-sequential preorder.
\end{proof}

Since ranking indices often map distinct fuzzy numbers to the same real value (lacking antisymmetry), they are merely total preorders. To break ties and increase the number of equivalence classes, researchers frequently {lexicographically} chain multiple indices. The sequential framework natively absorbs this technique for the construction of total preorders in any subfamily $\mathcal{S}$ of fuzzy numbers.

\begin{proposition}
Let $I_0, I_1, \dots, I_m$ be a finite family of ranking indices on a class $\mathcal{S}\subseteq\operatorname{FN}(\mathbb{R})$. The lexicographic preorder generated by sequentially applying these indices to break ties is a sequential preorder on $\mathcal{S}$.
\end{proposition}

Constructing a total order that is admissible (i.e., it refines the Klir-Yuan partial order $\le_{KY}$) is a complex task because it requires preserving the natural bounds of the fuzzy numbers while breaking ties across an uncountably infinite number of $\alpha$-levels. The sequential framework provides an elegant algebraic method to generate these admissible orders by utilizing a dense sequence of levels and a base space of intervals ($K=\operatorname{int}(\mathbb{R})$).

\begin{theorem}\label{th:seq_admis}
Let $\le_{\operatorname{int}}$ be an admissible total order on $\operatorname{int}(\mathbb{R})$ (meaning it refines the Kulisch-Miranker partial order $\le_{KM}$), and let $\{\alpha_n\}_{n \in \mathbb{N}_0}$ be an  upper dense sequence in $(0,1]$. Define the mapping $\Phi: \operatorname{FN}(\mathbb{R}) \to \operatorname{seq}(\operatorname{int}(\mathbb{R}))$ by:
\[ \Phi_{adm}(A) = \{A_{\alpha_0}, A_{\alpha_1}, A_{\alpha_2}, \dots\} \quad \text{for all } A \in \operatorname{FN}(\mathbb{R}), \]
where $A_{\alpha_n}$ is the $\alpha_n$-cut of $A$. Then, the $\Phi_{adm}$-sequential binary relation $\le_{\Phi_{adm}}$ is an admissible order on $\operatorname{FN}(\mathbb{R})$.
\end{theorem}

\begin{proof}
First, we prove that $\le_{\Phi_{adm}}$ is a total order. Because $\le_{\operatorname{int}}$ is a total order (it is antisymmetric), Corollary \ref{coro:Seq_Order} states that to prove $\le_{\Phi_{adm}}$ is a total order on $\operatorname{FN}(\mathbb{R})$, we only need to show that $\Phi_{adm}$ is injective. Let $A, B \in \operatorname{FN}(\mathbb{R})$ such that $\Phi_{adm}(A) = \Phi_{adm}(B)$. This implies $A_{\alpha_n} = B_{\alpha_n}$ for all $n \in \mathbb{N}_0$. Since the sequence $\{\alpha_n\}_{n \in \mathbb{N}_0}$ is upper dense in $(0,1]$, $A_{\alpha_n} = B_{\alpha_n}\pt n\in\N_0$, which implies that $A_\alpha = B_\alpha$ for all $\alpha \in [0,1]$. Thus, $A = B$, proving that $\Phi_{adm}$ is injective and, consequently, $\le_{\Phi_{adm}}$ is a total order.

Next, we prove {the} admissibility. Let $A, B \in \operatorname{FN}(\mathbb{R})$ be such that $A \le_{KY} B$. 
If $A = B$, then trivially $A \le_{\Phi_{adm}} B$ by reflexivity. 
Assume $A \neq B$. By the definition of the Klir-Yuan order, $A \le_{KY} B$ implies that for every $\alpha \in [0,1]$, the intervals satisfy $A_\alpha \le_{KM} B_\alpha$. Because $\le_{\operatorname{int}}$ is an admissible interval order, it refines $\le_{KM}$, which implies $A_\alpha \le_{\operatorname{int}} B_\alpha$ for all $\alpha \in [0,1]$. Since $A \neq B$ and $\Phi_{adm}$ is injective, the sequences $\Phi_{adm}(A)$ and $\Phi_{adm}(B)$ are not strictly equal. Let $n_0 \in \mathbb{N}_0$ be the minimum index such that $A_{\alpha_{n_0}} \neq B_{\alpha_{n_0}}$. 
For all $n < n_0$, we have $A_{\alpha_n} = B_{\alpha_n}$. For the index $n_0$, because we established $A_{\alpha_{n_0}} \le_{\operatorname{int}} B_{\alpha_{n_0}}$ and they are not equal, it must be strictly true that $A_{\alpha_{n_0}} <_{\operatorname{int}} B_{\alpha_{n_0}}$. By the definition of the sequential order $\sqsubseteq$ on $\operatorname{seq}(\operatorname{int}(\mathbb{R}))$, this condition exactly means $\Phi_{adm}(A) \sqsubseteq \Phi_{adm}(B)$. Therefore, $A \le_{\Phi_{adm}} B$. Since $A \le_{KY} B \implies A \le_{\Phi_{adm}} B$, the sequential order $\le_{\Phi_{adm}}$ is an admissible order on $\operatorname{FN}(\mathbb{R})$.
\end{proof}

We recall here that the abstract generalization given in Section \ref{sec:algebraic} provides a powerful tool for concrete applications. For instance, let $X = \operatorname{FN}(\mathbb{R})$ be the set of fuzzy numbers equipped with standard fuzzy addition $\oplus$ and non-negative scalar multiplication $\odot$ via Zadeh's extension principle. Let $K = \operatorname*{int}(\mathbb{R})$ with standard interval arithmetic and an admissible interval order. The admissible mapping $\Phi(A) = \{A_{\alpha_n}\}$ defined in Theorem \ref{th:seq_admis} strictly satisfies $\Phi(A \oplus B) = \Phi(A) +_{\operatorname*{seq}} \Phi(B)$ and $\Phi(\lambda \odot A) = \lambda \cdot_{\operatorname*{seq}} \Phi(A)$. Thus, by Theorem \ref{th_algebraic_compatibility} and its {homologous \ref{LM_algebraic_compatibility_2}}, the admissible order $\le_{\Phi_{adm}}$ is compatible with fuzzy arithmetic as long as the underlying interval order is translation invariant.

It should be remarked that the result of ranking fuzzy numbers directly depends on the used methodology. The main objective of the following example is to show how simple it is to consider total sequential preorders (and orders) when ranking fuzzy numbers, easily computing the resulting comparisons that resolve some incomparabilities.

\begin{example}
     Let consider the two trapezoidal fuzzy numbers $A$ and $B$ defined by their corners as $A = (1, 4, 5, 8)$ and $B = (2, 3, 6, 7)$. Note that $A$ and $B$ are incomparable under the Klir-Yuan partial order ($\le_{KY}$). While the support of $A$ ($[1, 8]$) contains the support of $B$ ($[2, 7]$), the core of $B$ ($[3, 6]$) contains the core of $A$ ($[4, 5]$), as shown in Figure \ref{fig:comparison}.

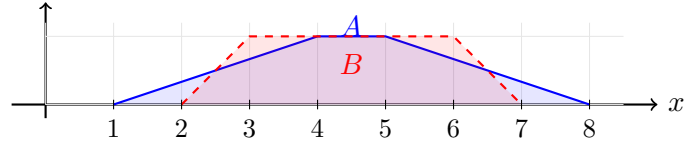
\begin{figure}[h]
    \centering
    \begin{tikzpicture}[scale=0.9]
        \draw[->, thick] (-0.5,0) -- (9,0) node[right] {$x$};
        \draw[->, thick] (0,-0.2) -- (0,1.5) node[above] {};
        
        \draw[help lines, gray!20] (0,0) grid (8.5,1.2);

        \draw[blue, thick] (1,0) -- (4,1) -- (5,1) -- (8,0);
        \fill[blue, opacity=0.1] (1,0) -- (4,1) -- (5,1) -- (8,0) -- cycle;
        \node[blue, font=\bfseries] at (4.5,1.15) {$A$};
        
        \draw[red, thick, dashed] (2,0) -- (3,1) -- (6,1) -- (7,0);
        \fill[red, opacity=0.1] (2,0) -- (3,1) -- (6,1) -- (7,0) -- cycle;
        \node[red, font=\bfseries] at (4.5,0.6) {$B$};
        
        \foreach \x in {1,2,3,4,5,6,7,8}
            \draw (\x, 2pt) -- (\x, -2pt) node[below, font=\small] {$\x$};
    \end{tikzpicture}
    \caption{Comparison of $A=(1,4,5,8)$ and $B=(2,3,6,7)$.}
    \label{fig:comparison}
\end{figure}

Let us {overcome this incomparability by employing two distinct sequential approaches}. First, we utilize an admissible order based on $\alpha$-cuts. Concretely, we will utilize the Lexicographical 2 interval order (which prioritizes the upper bound) to compare the supports of the trapezoidal fuzzy numbers and, if tied, we will compare the cores. Since $A_0 = [1, 8] > B_0 = [2, 7]$, this method yields {$B < A$}. Next, we will configure a sequential mapping based on four indices: (i) $I_0$ is the arithmetic mean of the vertices, (ii) $I_1$ is the upper bound of the core ($\alpha=1$), (iii) $I_2$ is the lower bound of the core, and (iv) $I_3$ is the lower bound of the support. Since the arithmetic mean of the vertices of both fuzzy numbers match, the first index $I_0$ does not discriminate the fuzzy numbers, and we need to check the second one. According to $I_1$, since $I_1(A)=5<I_1(B)=6$, then, the sequential order determines \textbf{$A < B$}.
\end{example}

\section{New ranking methods based on Sequential orderings}\label{sec:new}

While the preceding sections established the theoretical robustness and algebraic compatibility of the sequential framework, its true mathematical value lies in its generative capacity. Classical decision-making models are frequently hindered by the limitations of defuzzifications, leading to information loss and massive equivalence classes. By defining the sequence space and the mapping $\Phi$, the sequential binary relation overcomes these limitations, allowing us to construct entirely novel, highly discriminative ranking methods. In this section, we leverage this generative property to develop new ranking algorithms. We extend the framework to handle infinite countable families of indices, construct strict, computationally efficient total orders for parametric finite fuzzy numbers, and address the multidimensional uncertainty of IT2FN.

\subsection{Sequential orderings based on countable ranking indexes}

Classical ranking methods for fuzzy numbers typically rely on a single ranking index (such as the expected value, the area under the curve, or a specific distance metric). The fundamental limitation of single-index methods is the inevitable loss of information, as many distinct fuzzy numbers are mapped to the same real value, yielding a total preorder with very wide equivalence classes.

While some authors have proposed using a finite vector of indices to break ties \cite{Neres2023}, the generalized sequential framework allows us to push this concept to its analytical limit by employing a countably infinite family of ranking indices.

\begin{proposition}\label{prop:index_preorder}
	Let ${\mathcal{I}=}\{I_{n}\}_{n\in\mathbb{N}_{0}}$ be a countable family
	of ranking indices on $\operatorname{FN}(\mathbb{R})$. The lexicographic
	preorder generated by sequentially applying these indices to break ties, 
	{ given as the sequential order determined by the mapping $\Phi_{\mathcal{I}}:\operatorname{FN}(\mathbb{R}%
	)\rightarrow\operatorname{seq}(\mathbb{R})$ defined as:
	\[
	\Phi_{\mathcal{I}}(A)=\{I_{0}(A),I_{1}(A),\dots,I_{m}(A),\dots\},
	\]
	}
	is a sequential preorder.
\end{proposition}

\begin{proof}
	Let $A,B\in\operatorname{FN}(\mathbb{R})$. If we assume $A\leq_{\Phi
		_{\mathcal{I}}}B$, then $\Phi_{\mathcal{I}}(A)\sqsubseteq\Phi_{\mathcal{I}%
	}(B)$. By the definition of $\sqsubseteq$, either $\Phi_{\mathcal{I}}%
	(A)=\Phi_{\mathcal{I}}(B)$ (meaning all indices yield identical values for $A$
	and $B$), or there is a minimum index $k\in\mathbb{N}_{0}$ such that
	$I_{n}(A)=I_{n}(B)$ for all $n<k$ and $I_{k}(A)<I_{k}(B)$. This is precisely
	the formal definition of a lexicographic tie-breaking order applied to the
	countable set of ranking indices.
\end{proof}

The following result enables the generation of admissible preorders \cite{Centroid_Admis}.

\begin{theorem} \label{theor:index_admissible_preorder}
Let $\mathcal{I} = \{I_n\}_{n \in \mathbb{N}_0}$ be a countable sequence of ranking indices such that each $I_n: \operatorname{FN}(\mathbb{R}) \to \mathbb{R}$ is increasing on the $\alpha$-cuts, i.e., for any $A, B \in \operatorname{FN}(\mathbb{R})$, $A \le_{KY} B \implies I_n(A) \le I_n(B)$ for all $n \in \mathbb{N}_0$. Then, the sequential total preorder $\le_{\Phi_{\mathcal{I}}}$ is an admissible total preorder.
\end{theorem}

\begin{proof}
To prove that $\le_{\Phi_{\mathcal{I}}}$ is admissible, we must show that $A \le_{KY} B \implies A \le_{\Phi_{\mathcal{I}}} B$ for any $A, B \in \operatorname{FN}(\mathbb{R})$. 
Suppose $A \le_{KY} B$. By the hypothesis that each index $I_n$ is increasing on the $\alpha$-cuts, we have $I_n(A) \le I_n(B)$ for every $n \in \mathbb{N}_0$. 
This implies that the corresponding sequences in $seq(\mathbb{R})$ satisfy $\Phi_{\mathcal{I}}(A) \sqsubseteq \Phi_{\mathcal{I}}(B)$. 
By the definition of the sequential relation $\le_{\Phi_{\mathcal{I}}}$ given in Proposition \ref{prop:index_preorder}, this is equivalent to $A \le_{\Phi_{\mathcal{I}}} B$. Thus, the preorder is admissible.
\end{proof}

The true power of a countable sequence of indices lies in its potential to
completely characterize the fuzzy number, thereby transforming the total
preorder into a strict total order (an antisymmetric relation).

\begin{definition}
	A {countable} family of ranking indices
	$\mathcal{I}=\{I_{n}\}_{n\in\mathbb{N}_{0}}$ is said to be {separating} on a
	subfamily $\mathcal{S}\subseteq\operatorname{FN}(\mathbb{R})$ if, for any two
	{distinct} fuzzy numbers $A,B\in\mathcal{S}$,
	there exists at least $k\in\mathbb{N}_{0}$ such that
	$I_{k}(A)\neq I_{k}(B)$.
\end{definition}

\begin{theorem}\label{theor:index_total}
	If the {countable} family of ranking indices
	$\mathcal{I}=\{I_{n}\}_{n\in\mathbb{N}_{0}}$ is separating in $\mathcal{S}%
	\subseteq\operatorname{FN}(\mathbb{R})$, then the sequential mapping
	$\Phi_{\mathcal{I}}$ is injective and, consequently, the {countable }
	index-based sequential order $\leq_{\Phi_{\mathcal{I}}}$
	is a total order on $\operatorname{FN}(\mathbb{R})$.
\end{theorem}

\begin{proof}
	Let $A,B\in\mathcal{S}$ such that $A\equiv_{\Phi_{\mathcal{I}}}B$. This
	implies that $\Phi_{\mathcal{I}}(A)\equiv\Phi_{\mathcal{I}}(B)$, meaning
	$I_{n}(A)=I_{n}(B)$ for all $n\in\mathbb{N}_{0}$. Because $\{I_{n}%
	\}_{n\in\mathbb{N}_{0}}$ is a separating family, the condition $I_{n}%
	(A)=I_{n}(B)$ for all $n$ strictly requires that $A=B$. Thus, $\Phi
	_{\mathcal{I}}$ is injective. By Corollary \ref{coro:Seq_Order}, since
	$\Phi_{\mathcal{I}}$ injectively maps $\mathcal{S}$ into the totally ordered
	space $\operatorname{seq}(\mathbb{R})$, the induced relation $\leq
	_{\Phi_{\mathcal{I}}}$ is antisymmetric, making it a total
	 order.
\end{proof}

This formulation also behaves exceptionally well with the algebraic
compatibility theorems established before. If every ranking index $I_{n}$ is a
linear operator (i.e., $I_{n}(A\oplus B)=I_{n}(A)+I_{n}(B)$ and $I_{n}%
(\lambda\odot A)=\lambda I_{n}(A)$), then the entire mapping $\Phi
_{\mathcal{I}}$ is linear. Therefore, any total order constructed from a
separating sequence of linear ranking indices will automatically be compatible
with fuzzy addition and scalar multiplication. Furthermore, we can also obtain admissible orders utilizing ranking indices.

\begin{theorem} \label{theo:admissible_order}
Let $\mathcal{I} = \{I_n\}_{n \in \mathbb{N}_0}$ be a sequence of ranking indices that are increasing on the $\alpha$-cuts. Let $\mathcal{S} \subseteq \operatorname{FN}(\mathbb{R})$ be a subfamily of fuzzy numbers such that $\mathcal{I}$ is separating on $\mathcal{S}$. Then the sequential binary relation $\le_{\Phi_{\mathcal{I}}}$ is an admissible total order on $\mathcal{S}$.
\end{theorem}

\begin{proof}
From Theorem \ref{theor:index_total}, we know that if $\mathcal{I}$ is separating on $\mathcal{S}$, then $\le_{\Phi_{\mathcal{I}}}$ is a total order on $\mathcal{S}$. 
Furthermore, since the indices $I_n$ are increasing on the $\alpha$-cuts, the same argument used in the proof of Proposition \ref{theor:index_admissible_preorder} applies: 
For any $A, B \in \mathcal{S}$, $A \le_{KY} B$ implies $\Phi_{\mathcal{I}}(A) \sqsubseteq \Phi_{\mathcal{I}}(B)$, which means $A \le_{\Phi_{\mathcal{I}}} B$. 
Therefore, the relation $\le_{\Phi_{\mathcal{I}}}$ satisfies the conditions of an admissible order while maintaining the totality and antisymmetry guaranteed by Theorem \ref{theor:index_total}.
\end{proof}
\color{black}

\subsection{Sequential orderings for finite fuzzy numbers}

While fuzzy numbers given by continuous membership functions are prevalent in theoretical studies (triangular, trapezoidal, etc.), practical applications and computational algorithms heavily rely on discrete approximations. A prominent class of such approximations is the family of finite fuzzy numbers. In this section, we apply the sequential framework to this class, demonstrating how its parametric nature guarantees the existence of injective mappings and, consequently, strict total orders.

Following \cite{ROLDAN2014204}, a fuzzy number $A$ is {finite} if its image, $A(\mathbb{R})$, is a finite subset of $[0,1]$. In \cite{ALFONSO201747} it was shown that a fuzzy number $A$ is finite if and only if there is a partition $\Lambda = \{0 = \alpha_0 < \alpha_1 < \dots < \alpha_m = 1\}$ of the interval $[0,1]$ such that the $\alpha$-cuts of $A$ are common on each sub-interval $(\alpha_{i-1}, \alpha_i]$, i.e., $A_\alpha=A_\beta$ for all $\alpha,\beta\in(\alpha_{i-1}, \alpha_i]$ and all $ i \in \{1,...,m\}$. 
Let us denote by $\operatorname{FFN}_{\Lambda}(\mathbb{R})$ to the family of all finite fuzzy numbers associated to the partition $\Lambda$.


Because the $\alpha$-cuts are common over each sub-interval, any finite fuzzy number $A \in \operatorname{FFN}_{\Lambda}(\mathbb{R})$ is entirely and uniquely determined by the finite family of intervals corresponding to the partition levels: $\{A_{\alpha_1}, A_{\alpha_2}, \dots, A_{\alpha_m}\}$. Consequently, a finite fuzzy number is fundamentally a parametric fuzzy number characterized by $2m$ real numbers (the lower and upper bounds of these $m$ cuts).

This finite parametric representation is highly advantageous within our framework, as it allows for the construction of injective mappings.

\begin{proposition}
The mapping $\Phi_{param}: \operatorname{FFN}_{\Lambda}(\mathbb{R}) \to \operatorname{seq}(\mathbb{R})$ defined, for each $A \in \operatorname{FFN}_{\Lambda}(\mathbb{R})$ with $\alpha$-cuts $A_{\alpha_i}$ for $i \in \{ 1, \dots, m\}$, by:
\[ \Phi_{param}(A) = \{\underline{A}_{\,\alpha_1}, \overline{A}_{\alpha_1}, \underline{A}_{\,\alpha_2}, \overline{A}_{\alpha_2}, \dots, \underline{A}_{\,\alpha_m}, \overline{A}_{\alpha_m}, 0, 0, \dots\} \]
is injective. Consequently, the sequential order $\le_{\Phi}$ is a total  order on $\operatorname{FFN}_{\Lambda}(\mathbb{R})$.
\end{proposition}

\begin{proof}
Suppose $A, B \in \operatorname{FFN}_{\Lambda}(\mathbb{R})$ such that $\Phi_{param}(A) \equiv \Phi_{param}(B)$. This immediately implies that $\underline{A}_{\,\alpha_i} = \underline{B}_{\,\alpha_i}$ and $\overline{A}_{\alpha_i} = \overline{B}_{\alpha_i}$ for all $i \in \{1, \dots, m\}$. Since these $2m$ parameters completely define the finite fuzzy numbers on the partition $\Lambda$, it follows that $A = B$. Thus, $\Phi$ is injective. By Corollary \ref{coro:Seq_Order}, the induced total preorder is antisymmetric, making $\le_{\Phi}$ a strict total  order.
\end{proof}

While the raw parametric mapping $\Phi_{param}$ perfectly distinguishes finite fuzzy numbers, it may not reflect the heuristic preferences of a decision-maker (e.g., it arbitrarily compares lower bounds before upper bounds). To build a behaviorally meaningful ranking method, we can utilize aggregation functions to construct descriptive indices for each finite and chain them sequentially.

\subsection{Sequential orderings for Interval Type-2 Fuzzy Numbers}

The problem of ranking fuzzy numbers becomes significantly more complex when transitioning from standard type-1 fuzzy numbers to IT2FNs. Because IT2FNs incorporate a secondary layer of uncertainty (the so-called footprint of uncertainty), decision-makers must evaluate two distinct membership functions simultaneously. The generalized sequential framework developed in Section \ref{sec:sequential} provides a rigorous solution to this problem by natively absorbing multidimensional comparisons through interleaving sequences.

 A natural partial order on $\operatorname{IFN}_2(\mathbb{R})$
 extends the Klir--Yuan order by requiring dominance in both uncertainty bounds. For $A, B \in \operatorname{IFN}_2(\mathbb{R})$, we define:
\[ A \le_{KY2} B \iff \underline{A} \le_{KY} \underline{B} \quad \text{and} \quad \overline{A} \le_{KY} \overline{B}. \]

To construct a total order that refines this partial order, we can map IT2FNs into the sequence space of intervals by evaluating the $\alpha$-cuts of both the upper and lower memberships across a dense sequence of levels.

\begin{theorem}
\label{th_IT2_admissible}
Let $\le_{\operatorname{int}}$ be an admissible total order on $\operatorname{int}(\mathbb{R})$, and let $ \{\alpha_n\}_{n \in \mathbb{N}_0}$ be an upper dense sequence in $(0,1]$. Define the mapping $\Phi: \operatorname{IFN}_2(\mathbb{R}) \to \operatorname{seq}(\operatorname{int}(\mathbb{R}))$ by interleaving the $\alpha$-cuts of the upper and lower membership functions:
\[ \Phi(A) = \{\overline{A}_{\alpha_0}, \underline{A}_{\, \alpha_0}, \overline{A}_{\alpha_1}, \underline{A}_{\, \alpha_1}, \dots\} \quad \text{for all } A \in \operatorname{IFN}_2(\mathbb{R}). \]
Specifically, the sequence components are given by $\Phi(A)(2n) = \overline{A}_{\alpha_n}$ and $\Phi(A)(2n+1) = \underline{A}_{\,\alpha_n}$. 
Then, the $\Phi$-sequential binary relation $\le_{\Phi}$ is an admissible total order on $\operatorname{IFN}_2(\mathbb{R})$ with respect to $\le_{KY2}$.
\end{theorem}

\begin{proof}
First, we prove that $\le_{\Phi}$ is a total order by showing that $\Phi$ is injective. Let $A, B \in \operatorname{IFN}_2(\mathbb{R})$ such that $\Phi(A) = \Phi(B)$. This implies that for all $n \in \mathbb{N}_0$, $\overline{A}_{\alpha_n} = \overline{B}_{\alpha_n}$ and $\underline{A}_{\alpha_n} = \underline{B}_{\alpha_n}$. Since the sequence $ \{\alpha_n\}_{n \in \mathbb{N}_0}$ is upper dense in $(0,1]$, $\overline{A} = \overline{B}$ and $\underline{A} = \underline{B}$, which means $A = B$. Since $\Phi$ is injective into a totally ordered sequence space (Corollary \ref{coro:Seq_Order}), $\le_{\Phi}$ is a total order.

Next, we prove admissibility. Assume $A \le_{KY2} B$ and $A \neq B$. By definition, this means $\overline{A}_\alpha \le_{KM} \overline{B}_\alpha$ and $\underline{A}_\alpha \le_{KM} \underline{B}_\alpha$ for all $\alpha \in [0,1]$. Since $\le_{\operatorname{int}}$ is an admissible interval order, it refines $\le_{KM}$, yielding $\overline{A}_\alpha \le_{\operatorname{int}} \overline{B}_\alpha$ and $\underline{A}_\alpha \le_{\operatorname{int}} \underline{B}_\alpha$.
Since $A \neq B$, their sequences must be different. Let $k_0 \in \mathbb{N}_0$ be the minimum sequence index where $\Phi(A)(k) \neq \Phi(B)(k)$.
\begin{enumerate}
    \item If $k_0 = 2m$ (an even index), the divergence occurs at the upper membership function level $\alpha_m$. Since $\overline{A}_{\alpha_m} \le_{\operatorname{int}} \overline{B}_{\alpha_m}$ and they are not equal, we have $\overline{A}_{\alpha_m} <_{\operatorname{int}} \overline{B}_{\alpha_m}$.
    \item If $k_0 = 2m+1$ (an odd index), the divergence occurs at the lower membership function level $\alpha_m$. Since $\underline{A}_{\,\alpha_m} \le_{\operatorname{int}} \underline{B}_{\,\alpha_m}$ and they are not equal, we have $\underline{A}_{\,\alpha_m} <_{\operatorname{int}} \underline{B}_{\,\alpha_m}$.
\end{enumerate}
In both cases, at the first point of divergence, the component of $\Phi(A)$ is strictly less than the component of $\Phi(B)$. By the definition of the sequential order $\sqsubseteq$, this means $\Phi(A) \sqsubseteq \Phi(B)$. Therefore, $A \le_{KY2} B \implies A \le_{\Phi} B$, making it an admissible order on $\operatorname{IFN}_2(\mathbb{R})$.
\end{proof}

Note that the mapping $\Phi$ structurally prioritizes the upper membership function (the wider bounds of uncertainty) by placing $\overline{A}_{\alpha_n}$ before $\underline{A}_{\,\alpha_n}$. If a decision-maker wishes to prioritize the most conservative estimates (the lower membership function), the sequence can be trivially inverted to $\Phi^* (A) = \{\underline{A}_{\,\alpha_0}, \overline{A}_{\alpha_0}, \dots\}$, generating an equally valid, alternative total order.

\section{Conclusions}
\label{sec:conclusion}

In this paper, we have addressed the challenge of incomparability and information loss that limit classical ordering methods in fuzzy decision-making. To overcome these limitations, we have introduced a generalized mathematical framework based on totally preordered sequence spaces. By embedding complex fuzzy objects into this sequence space via an evaluation mapping, we established a lexicographical mechanism to resolve ties. 

Building upon this core theoretical foundation, we demonstrated the capacity of the sequential framework to act as a unifying algebraic umbrella for the broader literature. We provided mathematical proofs showing that highly technical relations (such as the $\alpha$-order and admissible orders based on upper dense sequences) are perfectly characterized as specific instances of the $\Phi$-sequential relation. Beyond unification, we explored the generative and practical potential of the framework by extending it to diverse computational environments. We have successfully provided the theoretical properties to construct total orders (even admissible) for fuzzy numbers based on real-valued ranking indexes; we developed strictly injective parametric mappings to construct computationally efficient total orders for finite fuzzy numbers; and we adapted the sequential approach to handle the multidimensional uncertainty of IT2FNs. 

Future research will focus on integrating these algebraic sequential algorithms into applied computational systems, such as fuzzy logic controllers and machine learning models, and extending the sequence space architecture to accommodate other advanced environments like intuitionistic or hesitant fuzzy sets.
\section*{Acknowledgements}

The authors are grateful to their universities.

\bibliographystyle{plain}

\end{document}